\documentclass[11pt]{amsart}
\usepackage{amsmath,amssymb,microtype}
\numberwithin{equation}{section}
\usepackage{xcolor}

\newtheorem{theorem}{Theorem}[section]
\newtheorem{lemma}[theorem]{Lemma}

\theoremstyle{definition}
\newtheorem{definition}[theorem]{Definition}
\newtheorem{example}[theorem]{Example}
\newtheorem{remark}[theorem]{Remark}
\newtheorem*{ackn}{Acknowledgements}

\newcommand \PSH {{\rm PSH}}
 
\subjclass[2010]{32W20, 32U05, 32Q15, 35A23}

\keywords{Complex Monge-Amp\`ere equations,  Mean field equations, Uniqueness, Stability, K\"ahler-Einstein metrics}

 \usepackage{hyperref}
\hypersetup{
    unicode=false,        
    pdftoolbar=true,      
    pdfmenubar=true,       
    pdffitwindow=false,     
    pdfstartview={FitH},    
    pdftitle={Complex Monge-Amp\`ere mean field equations},    
    pdfauthor={Lu, Phung},     
    colorlinks=true,       
   linkcolor=blue,         
    citecolor=blue,        
    filecolor=black,      
    urlcolor=blue}

\begin{document}

\title[Complex Monge-Amp\`ere mean field equations]{On uniqueness of solutions to  complex Monge-Amp\`ere mean field equations}

\author{Chinh H. Lu}

\address{Institut Universitaire de France \& Univ Angers, CNRS, LAREMA, SFR MATHSTIC, F-49000 Angers, France.}

\email{\href{mailto:hoangchinh.lu@univ-angers.fr}{hoangchinh.lu@univ-angers.fr}} 
\urladdr{\href{https://math.univ-angers.fr/~lu/}{https://math.univ-angers.fr/~lu/}}

\author{Trong-Thuc Phung}

\address{Ho Chi Minh City University of Technology, VNU-HCM, Viet Nam.}
\email{\href{mailto:ptrongthuc@hcmut.edu.vn}{ptrongthuc@hcmut.edu.vn}}
\urladdr{\href{https://fas.hcmut.edu.vn/personnel/ptrongthuc}{https://fas.hcmut.edu.vn/personnel/ptrongthuc}}

\date{\today}

 \begin{abstract}
We establish the uniqueness of solutions to complex Monge-Amp\`ere mean field equations when the temperature parameter is small. In the local setting of bounded hyperconvex domains, our result partially confirms a conjecture by Berman and Berndtsson. Our approach also extends to the global context of compact complex manifolds. 
\end{abstract}

\maketitle

%\tableofcontents

\section{Introduction}

%The complex Monge-Amp\`ere equations are important tools for studying canonical K\"ahler metrics, as shown in Yau's resolution of the Calabi conjecture \cite{Yau78}, and many recent achievements in K\"ahler geometry. 
It is classical that finding K\"ahler-Einstein metrics on a compact K\"ahler manifold $(X,\omega)$ of dimension $n$ is equivalent to solving complex Monge-Amp\`ere equations of the following type 
\begin{equation}
    \label{eq: MA compact intro}
    (\omega+dd^c \varphi)^n = e^{-\gamma \varphi} f\omega^n,
\end{equation}
where $f>0$ is a smooth function, $\gamma$ is a constant whose sign depends on that of the first Chern class of $X$. Here, we adopt the normalization in \cite{ACKPZ09}: $d= \partial + \bar \partial$ and $d^c = \frac{i}{2\pi}(\bar \partial -\partial)$, so that $dd^c =\frac{i}{\pi} \partial \bar \partial$.  When $\gamma \leq 0$, existence and uniqueness of smooth solutions was proved by Aubin \cite{Au78} (for $\gamma<0$) and Yau \cite{Yau78}. Using Ko{\l}odziej's estimate \cite{Kol98}, K\"ahler-Einstein metrics on singular K\"ahler varieties were constructed by Eyssidieux, Guedj, and Zeriahi \cite{EGZ09}. 
The uniqueness of solutions is well understood when $\gamma \leq 0$, even for very general right-hand sides (see \cite{DiwJFA09}, \cite{GZ07}). However, the situation is much more complicated when $\gamma > 0$, as solutions may not be unique, as  illustrated in an example of Fubini-Study metrics; see Example \ref{ex: FS}. Nonetheless, our first main result establishes uniqueness when $\gamma > 0$ is sufficiently small.

\begin{theorem}\label{thm: unique compact intro}
Assume $(X,\omega)$ is a compact K\"ahler manifold of dimension $n$, and $f$ is a probability density function on $X$ which belongs to $L^p(X,\omega^n)$, for some $p>1$. Then there exists $\gamma_0>0$, depending on $(X,\omega,n,p, \|f\|_p)$, such that for all $\gamma \in (0,\gamma_0)$, the equation \eqref{eq: MA compact intro} has a unique continuous solution.  
\end{theorem}

As mentioned above, the result in Theorem \ref{thm: unique compact intro} can not hold for all $\gamma>0$. Finding the critical value of $\gamma_0$ is thus an important problem. 
The proof of Theorem \ref{thm: unique compact intro} relies on a refinement of the stability estimate obtained in \cite{LPT21}, \cite{GLZAIF}, \cite{DZ10AIM}, \cite{Kol03IUMJ}. Using the same ideas we also obtain, in Section \ref{sect: hermitian}, similar results on compact Hermitian manifolds. 

Our argument also extends to provide a similar uniqueness result in the context of bounded hyperconvex domains, which we now describe. The equation under consideration is as follows:
\begin{equation}
    \label{eq: MAMF intro}
    (dd^c \varphi)^n = \frac{e^{-\gamma \varphi} fdV}{\int_{\Omega} e^{-\gamma \varphi} fdV}, \; \varphi \in \PSH(\Omega) \cap L^{\infty}, \; \varphi|_{\partial \Omega}=0,
\end{equation}
where $\Omega$ is a bounded hyperconvex domain in $\mathbb{C}^n$, $f$ is a probability density belonging to $L^p(\Omega,dV)$ for some $p>1$, and $\gamma>0$. 
When $f$ is smooth and positive, and $\Omega$ is strongly pseudoconvex, solutions of \eqref{eq: MAMF intro} are potentials of K\"ahler-Einstein metrics in $\Omega$ whose restrictions on the Levi distribution are conformal to the Levi form of $\partial \Omega$; see Section 2 of \cite{GKY13}. 
Using the variational method of \cite{BBGZ13}, it was shown in \cite{BB22} that \eqref{eq: MAMF intro} admits a continuous solution for $\gamma \in [0,\gamma_0)$, where $\gamma_0$ is a constant. Additionally, when $f$ is positive and smooth, as demonstrated in \cite{GKY13}, the solutions are smooth. Berman and Berndtsson conjectured that, when $\Omega$ is a Euclidean ball, any solution is radial, leading to the uniqueness of the solution; see \cite[Conjecture 7.6]{BB22}.  In the work of Guedj, Kolev, and Yeganefar \cite{GKY13}, the uniqueness of $S^1$-invariant solutions for circled domains was established, and it is expected to hold more generally; see the comments below \cite[Theorem C]{BGT23}. 

  Our second main result shows that \eqref{eq: MAMF intro} has a unique solution when $\gamma$ is sufficiently small, thus partially confirming the conjecture of Berman and Berndtsson

\begin{theorem}\label{thm: unique local intro}
Assume $\Omega$ is a bounded hyperconvex domain in $\mathbb{C}^n$ and $f\geq 0$ is a probability density belonging to $L^p(\Omega,dV)$, for some $p>1$. 
There exists $\gamma_0>0$, depending on $n,p,\|f\|_p$ and the diameter of $\Omega$, such that for all $\gamma<\gamma_0$ the equation \eqref{eq: MAMF intro}
has a unique solution $u\in \PSH(\Omega)\cap L^{\infty}(\Omega)$. 
\end{theorem}

In the one-dimensional case, equation \eqref{eq: MAMF intro} is referred to as the mean field equation, as explained in \cite[Section 7.3]{BB22}, where  $\gamma$ corresponds to (minus) the temperature parameter. In this case, when $f>0$ is smooth, uniqueness of solutions was recently established in \cite{GM18}, with related results in \cite{SSTW19}, \cite{BL14}. While uniqueness is well understood in one dimension, it remains a challenging open problem in higher dimensions. By \cite{Kol98}, when $\Omega$ is strongly pseudoconvex, the solution is automatically continuous, see \cite{Kol98}. 
Additionally, we highlight the recent work of Badiane and Zeriahi, where a similar uniqueness result was obtained via the eigenvalue problem (see \cite[Proposition 4.2]{BZ23}). Our proof relies on the $L^{\infty}$ estimate of Ko{\l}odziej which now has several different proofs (see \cite{GPT23}, \cite{GL21a} as well as Theorem \ref{thm: Linfty est local} and \ref{thm: Linfty global}). As a result, our approach applies without requiring any smoothness assumptions on the boundary of $\Omega$ and extends to densities in $L^p(\Omega,dV)$, $p>1$, that are not necessarily smooth. Finally, our approach can be extended to handle arbitrary continuous boundary data, rather than just zero.

\medskip

 We prove the main results in Sections \ref{sect: hyperconvex} and \ref{sect: Kahler}. In Section \ref{sect: hermitian}, we extend our findings to the context of compact Hermitian manifolds. Finally, in an effort to make the constant $\gamma_0$ explicit, we propose a new method for proving the $L^{\infty}$ estimate for complex Monge-Amp\`ere equations in Theorems \ref{thm: Linfty est local} and  \ref{thm: Linfty global}.

\begin{ackn} 
We acknowledge support from the Institut Universitaire de France, the project PARAPLUI ANR-20-CE40-0019, the KRIS project funded by the fondation Charles Defforey, and the Centre Henri Lebesgue ANR-11-LABX-0020-01. We thank Bo Berndtsson for insightful conversations on mean field equations and are grateful to Vincent Guedj and Ahmed Zeriahi for numerous helpful discussions during the preparation of the paper.  
\end{ackn}

\section{Monge-Amp\`ere mean field equations in hyperconvex domains}\label{sect: hyperconvex}

In this section, $\Omega$ is a bounded hyperconvex domain in $\mathbb{C}^n$. By this we mean there exists a continuous function $u : \overline{\Omega} \rightarrow \mathbb{R}$ such that $u=0$ on $\partial \Omega$, $u$ is plurisubharmonic (psh for short) in $\Omega$, and for all $c>0$, the set $\{u<-c\}$ is relatively compact in $\Omega$. We let $\PSH(\Omega)$ denote the set of all psh functions in $\Omega$.  

For a bounded psh function $u$, the complex Monge-Amp\`ere operator $(dd^c u)^n$ is a positive Radon measure, as defined by Bedford and Taylor \cite{BT76}. Moreover, this operator is continuous along sequences converging in capacity, and in particular along monotone sequences. 

We next recall the definition of Cegrell's classes. 
Let $u$ be a negative psh function defined in $\Omega$. 

\begin{definition}[The class $\mathcal{E}_p$, $\mathcal{F}$, and $\mathcal{T}_0$]{\;}
\begin{enumerate}
    \item 
  We say that $u\in \mathcal{E}_0(\Omega)$ if $u$ is bounded, it vanishes on the boundary $\partial \Omega$, and its Monge-Amp\`ere measure has finite total mass.  

  \item We say that $u\in \mathcal{E}_p(\Omega)$, $p>0$, if there exists a decreasing sequence $(u_j)\subset \mathcal{E}_0(\Omega)$ that converges to $u$ and satisfies  
\[
\sup_j \int_{\Omega} |u_j|^p(dd^c u_j)^n <+\infty.
\]
  
 \item We say that $u\in \mathcal{F}(\Omega)$ if there exists a decreasing sequence $(u_j)\subset \mathcal{E}_0(\Omega)$ that converges to $u$ and satisfies
\[
\sup_j \int_{\Omega} (dd^c u_j)^n <+\infty.
\]

\item We let $\mathcal{T}_0$ be the set of all $u\in \mathcal{E}_0(\Omega)$ such that $\int_{\Omega}(dd^c u)^n \leq 1$. 
\end{enumerate}
\end{definition}

As demonstrated by Cegrell \cite{Ceg98}, \cite{Ceg04}, the complex Monge-Amp\`ere operator can be conveniently defined for functions in $\mathcal{F}(\Omega)$ and $\mathcal{E}_p(\Omega)$, even when these functions are unbounded. Furthermore, for any positive non-pluripolar measure $\mu$ with finite total mass, the equation $(dd^c u)^n = \mu$ has a unique solution in $\mathcal{F}(\Omega)$. 

 Let $dV$ denote the Euclidean volume form of $\mathbb{C}^n$ and fix a probability density $0\leq f\in L^p(\Omega,dV)$, for some $p>1$. We study the following Monge-Amp\`ere equation 
\begin{equation}
	\label{eq: MAMF}
	(dd^c u)^n = \frac{e^{-\gamma u}fdV }{ \int_{\Omega} e^{-\gamma u} fdV}, \; u\in \PSH(\Omega)\cap L^{\infty}, \; u|_{\partial \Omega}=0,
\end{equation}
where $\gamma$ is a positive constant.  As shown in \cite{BB22} there exists $\gamma_0>0$ such that for all $\gamma\in (0,\gamma_0)$, the equation \eqref{eq: MAMF} admits a continuous solution, which is smooth up to the boundary when $f>0$ is smooth and $\Omega$ is strongly pseudoconvex; see \cite{GKY13}. Our goal in this section is to investigate the uniqueness problem. Note that solutions of \eqref{eq: MAMF} are automatically normalized with total Monge-Amp\`ere mass equal to $1$. In our analysis, we will frequently use the comparison principle, which states that subsolutions lie below supersolutions. For this reason, it will be more convenient to consider the following non-normalized equation
\begin{equation}
	\label{eq: MAMF non normalized}
	(dd^c u)^n = e^{-\gamma u +m}fdV,  \; \; u\in \PSH(\Omega)\cap L^{\infty}, \; u|_{\partial \Omega}=0,
\end{equation}
where $m$ is a fixed constant. 

\begin{lemma}[Maximal solution]\label{lem: maximal sol}
	If \eqref{eq: MAMF non normalized} admits a subsolution, then it admits a maximal solution. 
\end{lemma}

\begin{proof}
A subsolution of \eqref{eq: MAMF non normalized} is a bounded plurisubharmonic  function vanishing on the boundary such that the inequality $\geq$ holds. 	Assume $v,w\in \PSH(\Omega) \cap L^{\infty}(\Omega)$ are subsolutions of  \eqref{eq: MAMF non normalized}. Then, by the maximum principle, see \cite[Corollary 3.28]{GZbook}, we have 
	\begin{flalign*}
	(dd^c \max(v,w))^n & \geq 	{\bf 1}_{\{v\geq w\}} (dd^c v)^n + {\bf 1}_{\{v<w\}} (dd^c w)^n \\
	& \geq  	{\bf 1}_{\{v\geq w\}} e^{-\gamma v+m} fdV + {\bf 1}_{\{v<w\}} e^{-\gamma w +m} fdV\\
	& = e^{-\gamma \max(v,w)+m} fdv.  
	\end{flalign*}	

Let $\mathcal S$ denote the set of all bounded subsolutions of \eqref{eq: MAMF non normalized} and let $u$ be the upper envelope of $\mathcal S$. By this we mean 
\[
u := \left ( \sup \{v \in \mathcal S\} \right )^*,  
\]
the latter being well-defined because the family is bounded from above by $0$.
By Choquet's lemma there exists an increasing sequence $(u_j)\subset \mathcal S$, such that $(\lim u_j)^*=u$. Since $\mathcal S$ is stable under taking maximum, we infer from the continuity of the complex Monge-Amp\`ere operator along increasing sequences that $u\in \mathcal S$.

We next inductively define the sequence $(\psi_k)$ by solving the complex Monge-Amp\`ere equations 
	\[
	(dd^c \psi_{j+1})^n = e^{-\gamma \psi_k +m} fdV, \; \psi_{k+1}|_{ \partial \Omega}=0, \; \psi_0=u.  
	\]
    From  
    \[
    (dd^c \psi_0)^n \geq e^{-\gamma \psi_0 +m} fdV,\; (dd^c \psi_1)^n = e^{-\gamma \psi_0 +m} fdV,
    \]
	and the comparison principle, \cite[Corollary 3.30]{GZbook}, we obtain $\psi_0 \leq \psi_1$. Repeating this reasoning, we see that  $(\psi_k)$ is increasing with respect to $k$. Since $\psi_k\leq 0$, the sequence increasingly converges almost everywhere to some $\psi \in \PSH(\Omega)\cap L^{\infty}$. By the continuity of the complex Monge-Amp\`ere operator,  see \cite[Theorem 3.23]{GZbook}, it follows that $\psi$ solves \eqref{eq: MAMF non normalized}, hence $\psi \leq u$. However, the reverse inequality $u \leq \psi$ also holds because $u=\psi_0\leq \psi_k$. We thus infer that $u$ is the maximal solution of the equation \eqref{eq: MAMF non normalized}. 
\end{proof}

\begin{lemma}\label{lem: uniqueness for small sol}
	Assume $u,v\in \mathcal{F}(\Omega)$ solve  \eqref{eq: MAMF} and $\gamma \sup_{\Omega}|u|<n$, $\gamma \sup_{\Omega}|v|<n$. Then $u =v$.
\end{lemma}

\begin{proof}
Assume by contradiction that $u\neq v$. 
	Since they play the same role, we can assume that $m_u\geq m_v=m$, where $m_u= -\log \int_{\Omega}  e^{-\gamma u} fdV$. We then have 
	\[
	(dd^c \max(u,v))^n \geq e^{-\gamma \max(u,v) +m} fdV.
	\]
	By Lemma \ref{lem: maximal sol}, the maximal solution $w$ of the equation 
	\[
	(dd^c w)^n = e^{-\gamma w + m}  fdV,\; w|_{\partial \Omega} =0,
	\]
	satisfies $v\leq w$ and $v\neq w$. We let $\rho \in \PSH(\Omega)\cap L^{\infty}$ solve 
	\[
	(dd^c \rho)^n = \left (e^{(-\gamma v+m)/n} - e^{(-\gamma w+m)/n} \right )^n fdV = gdV, \; \rho|_{\partial \Omega =0}. 
	\]
    Using 
	\[
	e^y - e^x \leq (y-x) e^y, \; \forall x\leq y \in \mathbb R,
	\]
	we have 
	\[
	g \leq  n^{-n}\gamma^n (w-v)^n e^{-\gamma v+m} f. 
	\]
	By the comparison principle, see \cite[Corollary 3.30]{GZbook}, we thus get 
	\[
	\rho \geq \left ( n^{-1} \gamma \sup_{\Omega} |v-w|  \right) v,
	\]
	and hence $|\rho| \leq \delta \sup_{\Omega}|v-w|$,  where $\delta = n^{-1}\gamma \sup_{\Omega}|v|$, and by the assumption $\gamma \sup_{\Omega}|v|<n$, we have that $\delta\in (0,1)$.  Moreover, the mixed Monge-Amp\`ere inequalities (see Lemma \ref{lem: mixed MA} below) give 
	\[
	(dd^c (w + \rho))^n \geq (g^{1/n} + e^{(-\gamma w+m)/n} f^{1/n})^n = e^{-\gamma v+m} fdV = (dd^c v)^n.
	\]
	Applying again the comparison principle we obtain $w+\rho \leq v$, therefore 
	\[
	0\leq w-v \leq |\rho| \leq \delta\sup_{\Omega} |w-v|,
	\]
	which forces $w=v$, a contradiction.  
\end{proof}

The following result follows directly from the mixed Monge-Amp\`ere inequalities (see \cite{Gar59}, \cite{Kol03IUMJ}, \cite{Diw09Z}). 
\begin{lemma}[Mixed Monge-Amp\`ere inequalities]\label{lem: mixed MA}
    Assume $u,v$ are bounded psh functions and $g,h$ are $L^1$ densities such that 
    \[
    (dd^c u)^n \geq g dV, \; (dd^c v)^n \geq hdV.
    \]
    Then 
    \[
    (dd^c (u+v))^n \geq (g^{1/n} + h^{1/n})^n dV.  
    \]
\end{lemma}

\begin{proof}
    By the mixed Monge-Amp\`ere inequalities, see \cite{Gar59}, \cite{Kol03IUMJ}, \cite{Diw09Z}, we have 
    \[
    (dd^c u)^k \wedge (dd^c v)^{n-k} \geq g^{k/n} h^{(n-k)/n} dV.
    \]
    Thus the binomial expansion gives
    \begin{flalign*}
         (dd^c (u+v))^n &= \sum_{k=0}^n \binom{n}{k} (dd^c u)^k \wedge (dd^c v)^{n-k} \\
         &\geq \sum_{k=0}^n \binom{n}{k} g^{k/n} h^{(n-k)/n}dV \\
         &=(g^{1/n}+h^{1/n})^n dV. 
    \end{flalign*}
\end{proof}

We next move on to prove that for sufficiently small $\gamma>0$, the equation \eqref{eq: MAMF} admits a unique solution. In an effort to  make the constant $\gamma$ explicit, we provide the following $L^{\infty}$ estimate whose proof uses the resolution of \eqref{eq: MAMF}. 
\begin{theorem}\label{thm: Linfty est local}
    Assume $\gamma$ is a positive constant and $\mu$ is a probability measure in $\Omega$ which vanishes on pluripolar sets and satisfies
    \[
    A_{\mu}:= \sup \left \{ \int_{\Omega} e^{-\gamma u} d\mu \; : \; u \in \mathcal{T}_0 \right \}<+\infty. 
    \]
    If $\varphi \in \mathcal{F}(\Omega)$ solves $(dd^c \varphi)^n =\mu$, then 
    \[
    \varphi \geq -n A_{\mu}^{1/n} \gamma^{-1}.
    \]
\end{theorem}
Recall that $\mathcal{T}_0$ is the set of all functions $u\in \mathcal{E}_0(\Omega)$ such that $\int_{\Omega} (dd^c u)^n \leq 1$. 
\begin{proof}
     The variational approach can be employed (see \cite{BGT23}, \cite{BB22}), giving a solution $u\in \mathcal{E}_1(\Omega)$ of the equation 
    \[
    (dd^c u)^n = \frac{e^{-\gamma u}\mu}{\int_{\Omega}e^{-\gamma u}d\mu}. 
    \]
    The function $v=e^{\gamma u/n}-1$ is psh in $\Omega$. Moreover, $-1\leq v\leq 0$, and  
    \[
    dd^c v \geq \gamma n^{-1} e^{\gamma u/n} dd^c u,
    \]
    thus
    \[
    (dd^c v)^n \geq  \frac{\gamma^n}{n^n} e^{\gamma u} (dd^c u)^n \geq \frac{\gamma^n}{n^n A_{\mu}} \mu = (dd^c \psi)^n,
    \]
    where $\psi= n^{-1}A_{\mu}^{-1/n}\gamma \varphi$. By the comparison principle, \cite[Theorem 5.15]{Ceg04}, we thus have $\psi \geq v\geq -1$, yielding
    \[
    \varphi \geq -n A_{\mu}^{1/n} \gamma^{-1}.
    \]
\end{proof}

We are now ready to prove the main result of this section, demonstrating the uniqueness of the solution to mean field equations with a small temperature parameter.

\begin{theorem}\label{thm: unique local}
    Assume that, for some $\beta>0$ and $A\geq 1$ we have, 
    \[
 \int_{\Omega} e^{-\beta u} fdV \leq A, \; \forall u \in \mathcal{T}_0.
    \]
    Then for all $\gamma<\gamma_0:=  2^{-1}\beta A^{-1/n}$, \eqref{eq: MAMF} admits a unique solution. 
\end{theorem}

\begin{proof}
%By \cite{ACKPZ09}, there exists a constant $A$, depending on $n$ and the diameter of $\Omega$ such that, 
%\[
%\int_{\Omega} e^{-nu} 
%\]
    Assume that $\varphi\in \mathcal{F}(\Omega)$ solves \eqref{eq: MAMF} with parameter $\gamma<\gamma_0$. 
     It follows from \cite[Corollary 5.6]{Ceg04} that the set $\mathcal T_0$ is convex. 
      For $v\in \mathcal T_0$, we have
     \[
     \int_{\Omega}e^{-\beta v/2}(dd^{c}\varphi)^{n}\leq \frac{\int_{\Omega}e^{-\beta(v+\varphi)/2}fdV}{\int_{\Omega}e^{-(\beta/2)\varphi}fdV}.
     \]
     Observe that the denominator above is greater than $1$ because $\varphi\leq 0$ and $fdV$ is a probability measure. Since $(v+\varphi)/2 \in \mathcal{T}_0$, it thus follows that  
    \[
     \int_{\Omega}e^{-\beta v/2}(dd^{c}\varphi)^{n}\leq \int_{\Omega}e^{-\beta(v+\varphi)/2}fdV \leq A.
     \]
    It thus follows from Theorem \ref{thm: Linfty est local} that $\varphi \geq -2n A^{1/n} \beta^{-1}$. Since $\gamma<2^{-1}\beta A^{-1/n}$, any solution $\varphi$ of \eqref{eq: MAMF} satisfies $\gamma \sup_{\Omega} |\varphi|<n$. The uniqueness of the solution of \eqref{eq: MAMF} thus follows from Lemma \ref{lem: uniqueness for small sol}. 
\end{proof}

\subsection*{Proof of Theorem \ref{thm: unique local intro}} It follows from \cite[Theorem B]{ACKPZ09} that, for any $\alpha<n$, 
\[
\int_{\Omega}e^{-2\alpha u}dV \leq A_1, \; \forall u \in \mathcal{T}_0,
\]
where $A_1$ depends on $\alpha$, $n$, and the diameter ${\rm diam}(\Omega)$ of $\Omega$. Let $q$ be the conjugate exponent of $p$, i.e. $1/p +1/q=1$. Then, by H\"older's inequality, we have
\[
\int_{\Omega}e^{-2q^{-1}\alpha u}fdV \leq \left(\int_{\Omega} e^{-2\alpha u} dV\right)^{1/q} \|f\|_p.
\]
Thus, in Theorem \ref{thm: unique local}, taking $\beta=2q^{-1}\alpha$ we see that the constant $A$ only depends on $n, p, \|f\|_p$, and the diameter of $\Omega$. Therefore, taking e.g. $\alpha=n/2$, we see that $\gamma_0$ in Theorem \ref{thm: unique local} only depends on $n,p,\|f\|_p$, and ${\rm diam}(\Omega)$.

\section{Stability and uniqueness on compact K\"ahler manifolds}\label{sect: Kahler}

We fix a compact complex manifold $X$ of dimension $n$, and let $\omega$ be a fixed K\"ahler metric on $X$, normalized so that $\int_X \omega^n=1$.

Recall that quasi-plurisubharmonic (quasi-psh for short) functions are locally the sum of a psh function and a smooth function. A function $u$ is $\omega$-psh if it is quasi-psh and $\omega+dd^c u \geq 0$ in the weak sense of currents. We let $\PSH(X,\omega)$ denote the set of all $\omega$-psh functions that are integrable over $X$, or equivalently those that are not identically $-\infty$. 

For any bounded $\omega$-psh function $u$, following Bedford and Taylor \cite{BT76}, the Monge-Amp\`ere operator $(\omega+dd^c u)^n$ is well-defined as a positive Radon measure with maximal total mass 
\[
\int_X (\omega+dd^c u)^n =\int_X \omega^n,
\]
as demonstrated in \cite{GZ05}. For unbounded $\omega$-psh functions $u$, the non-pluripolar Monge-Amp\`ere measure $(\omega+dd^c u)^n$ is defined, in \cite{GZ07}, as a positive Radon measure: 
\[
(\omega+dd^c u)^n:= \lim_{j\to +\infty} {\bf 1}_{\{u>-j\}}(\omega+dd^c \max(u,-j))^n,
\]
where the sequence of measures increases with $j$. 
If this measure has maximal total mass then $u$ belongs to the class $\mathcal{E}(X,\omega)$. Functions in $\mathcal{E}$ are generally unbounded but their singularities are mild; for any $u\in \PSH(X,\omega)$ with $u\leq -1$, the function $-(-u)^a$ belongs to $\mathcal{E}(X,\omega)$, for any $0<a<1$. 

The class $\mathcal{E}^1(X,\omega)$ consists of all $u\in \mathcal{E}(X,\omega)$ that satisfy the integrability condition: 
\[
\int_X |u| (\omega+dd^c u)^n <+\infty. 
\]

The goal of this section is to prove Theorem \ref{thm: unique compact intro}. To begin, we establish two stability results which are refinements of \cite[Theorem 1.3]{LPT21}. These results build on the  ideas introduced in \cite{GLZAIF} and \cite{Kol03IUMJ}.

\begin{theorem}\label{thm: stability exp}
	Fix two constants $p>1$ and $B> 1$. There exists a constant $C$ depending on $p, B, \omega,n,X$, such that the following holds: if $f,g$ are nonnegative functions with 
	\[
	B^{-1}\leq  \|f\|_1 \leq \|f\|_p \leq B, \; B^{-1}\leq \|g\|_1 \leq \|g\|_p \leq B,
	\]
	 then the solutions $u$, $v$ of 
	\[
	(\omega+dd^c u)^n = e^u f\omega^n\; , \; (\omega+dd^c v)^n = e^v g\omega^n
	\]
	satisfy
	\[
	\sup_X |u-v| \leq C \|f^{1/n}-g^{1/n}\|_{np}. 
	\]
\end{theorem}

\begin{proof}
	Let $u_0,v_0$ be the unique normalized solutions to 
    \[
    (\omega+dd^c u_0)^n = c_ff \omega^n, \; \text{and}\; (\omega+dd^c v_0)^n = c_gg \omega^n,\; \sup_X u_0=\sup_X v_0=0.
    \]
    Here, $c_f, c_g$ are positive constants ensuring that the measures $c_f f\omega^n$ and $c_gg\omega^n$ are probability measures. More precisely, $c_f \int_X f\omega^n=c_g \int_X g\omega^n=1$. From the estimates 
    \[
    B^{-1}\leq \int_X f \omega^n \leq B \; \text{and}\; B^{-1}\leq \int_X g \omega^n \leq B,
    \]
    we infer that $B^{-1}\leq c_f \leq B$ and $B^{-1}\leq c_g \leq B$. 
    
    By Ko{\l}odziej's $L^{\infty}$-estimate, see \cite{Kol98}, $|u_0|\leq C_0$ and $|v_0|\leq C_0$, for a uniform constant $C_0$. We thus have 
    \[
    (\omega+dd^c u_0)^n \geq  e^{u_0 + \log c_f} f\omega^n, \;  (\omega+dd^c u_0)^n \leq  e^{u_0 +C_0+ \log c_f} f\omega^n. 
    \] 
    The domination principle, \cite[Proposition 10.11]{GZbook}, then gives 
    \[
    u_0+ \log c_f \leq u \leq u_0 + C_0+\log c_f.
    \]
    Similarly, we also have 
    \[
    v_0+ \log c_g\leq v \leq v_0 +C_0+ \log c_g,
    \]
    hence $\sup_X |u| \leq C_0+ \log B$ and $\sup_X |v|\leq C_0+\log B = :C_1$.

	To proceed, we can assume that 
	\[
	\varepsilon:= e^{C_1/n} \|f^{1/n}-g^{1/n}\|_{np} \in (0,1/2).
	\] 
	Let $\psi$ be the unique bounded $\omega$-psh function solving 
	\[
	(\omega+dd^c \psi)^n = \left(  \frac{e^v |f^{1/n}-g^{1/n}|^n}{e^{C_1} \||f^{1/n}-g^{1/n}|^n\|_p} + a\right ) \omega^n = h\omega^n, \; \sup_X \psi=0,
	\]
	where $a\geq 0$ is a constant ensuring that $h$ is a probability density. By the triangle inequality for the $L^p$ norm, we have that $\|h\|_p\leq 2$, therefore, $\sup_X |\psi| \leq C$.  We next consider $\varphi:= (1-\varepsilon)v + \varepsilon \psi -C_2 \varepsilon$, where $C_2= \sup_X (\psi-v)+2n$. By the mixed Monge-Amp\`ere inequalities, see \cite{Diw09Z},  we have 
	\begin{flalign*}
		(\omega+dd^c \varphi)^n & = ((1-\varepsilon) \omega_v + \varepsilon \omega_{\psi})^n \\
		& = \sum_{k=0}^n \binom{n}{k} (1-\varepsilon)^k \varepsilon^{n-k}\omega_v^k \wedge \omega_{\psi}^{n-k} \\
		& \geq \sum_{k=0}^n \binom{n}{k} (1-\varepsilon)^k e^{kv/n} g^{k/n} e^{(n-k)v/n} |f^{1/n}-g^{1/n}|^{n-k}  \omega^n\\
		& = e^v \left ( (1-\varepsilon)  g^{1/n} +  |f^{1/n}-g^{1/n}|  \right)^n \omega^n\\
		&\geq e^{v+ n\log (1-\varepsilon)} f \omega^n.  
	\end{flalign*}
Noting that $\log (1-\varepsilon) \geq -2\varepsilon $, we obtain 
	\[
	(\omega+dd^c \varphi)^n \geq e^{\varphi} f\omega^n.
	\] 
	By the domination principle we thus get $\varphi \leq u$, and this yields 
	\begin{flalign*}
		u & \geq (1-\varepsilon)v + \varepsilon \psi -C_2 \varepsilon \\
		&\geq v + \varepsilon \inf_X (\psi-v)  - C_2 \varepsilon\\
		& = v - ({\rm Osc}_X(\psi-v) + 2n) \varepsilon. 
	\end{flalign*}
	We can similarly establish the other bound, finishing the proof. 
\end{proof}

\begin{theorem}\label{thm:stability}
	Fix two constants $p>1$ and $B>1$. There exists a constant $C$ depending on $p, B, \omega,n,X$, such that the following holds: if $f,g$ are probability densities with $\|f\|_p \leq B$ and $\|g\|_p\leq B$, then the normalized solutions $u$, $v$ to 
	\[
	(\omega+dd^c u)^n = f\omega^n\; , \; (\omega+dd^c v)^n = g\omega^n, \; \sup_X u= \sup_X v =0,
	\]
	satisfy
	\[
	\sup_X |u-v| \leq C \|f^{1/n}-g^{1/n}\|_{np}. 
	\]
\end{theorem}

\begin{proof}
We shall use $C_0,C_1,...$ to denote various uniform constants. 
	As mentioned above, Ko{\l}odziej's $L^{\infty}$-estimate \cite{Kol98} ensures that 
	\[
	\max(\sup_X |u|, \sup_X |v|) \leq C_0. 
	\]

	We divide the proof into two steps. In the first step
	we assume 
	\[
	2^{-\varepsilon} f \leq g \leq 2^{\varepsilon }f,
	\] 
	for some $\varepsilon\in [0,1/2)$, and we prove that 
	\[
	|u-v| \leq C \varepsilon,
	\]
	for a uniform constant $C>0$. If we define 
	\[
	2\alpha = \sup_X (u-v) - \sup_X (v-u), 
	\]
	then we either have 
	\[
	\int_{\{u<v+\alpha\}} \omega^n \leq 1/2 \; \text{or}\; \int_{\{v<u-\alpha\}} \omega^n \leq 1/2.
	\]
	We assume the first inequality holds (a similar argument applies for the second).  Now, let $\psi\in \PSH(X,\omega)\cap L^{\infty}(X)$ be the unique solution of
	\[
	(\omega+dd^c \psi)^n = 2 {\bf 1}_E f\omega^n + b \omega^n, \; \sup_X \psi=0.
	\]
	Here $b\in [0,1]$ is a constant ensuring that the right-hand side is a probability measure, and to simplify the notation, we denote $E= \{u<v+\alpha\}$. The $L^p$ norm of the right-hand side above is smaller than $2 B +1$, thus Ko{\l}odziej's estimate gives $|\psi| \leq C_1$.  Observe that $\rho:= (1-\varepsilon)v + \varepsilon \psi$ is bounded, $\omega$-psh. Using the mixed Monge-Amp\`ere inequalities we obtain 
	\[
	(\omega+dd^c \rho)^n \geq ((1-\varepsilon) g^{1/n} + \varepsilon 2^{1/n} {\bf 1}_E f^{1/n})^n\omega^n. 
	\]
	Thus, 
	\begin{flalign*}
		{\bf 1}_E(\omega+dd^c \rho)^n & \geq {\bf 1}_E ((1-\varepsilon) g^{1/n} + \varepsilon 2^{1/n}  f^{1/n})^n\omega^n\\
		&\geq {\bf 1}_E  ((1-\varepsilon) 2^{-\varepsilon/n} + \varepsilon 2^{1/n})^n f\omega^n. 
	\end{flalign*}
	Using $2^x = e^{x\log 2} \geq 1+ x \log 2$, we continue as follows 
	\begin{flalign*}
		{\bf 1}_E(\omega+dd^c \rho)^n & \geq {\bf 1}_E  (1+ n^{-1}\varepsilon^2 \log 2)^n f\omega^n. 
	\end{flalign*}
    For $C_2:= \sup_X (\psi-v)$, we have 
	\[
	\{u <\rho+\alpha-C_2\varepsilon\} \subset \{u < v +\alpha + \varepsilon (\psi-v) -C_2 \varepsilon\}  \subset \{u <v+\alpha\}=E,
	\]
 and the above estimate can be reformulated as 
	  \[
	  (1+ n^{-1} \varepsilon^2 \log 2)^n {\bf 1}_{\{u<\rho +\alpha -C_2 \varepsilon\}} (\omega+dd^c u)^n \leq  {\bf 1}_{\{u<\rho +\alpha -C_2 \varepsilon\}} (\omega+dd^c \rho)^n. 
	  \]
	  Since the multiplicative constant on the left-hand side above is strictly larger than $1$,  the domination principle, \cite[Proposition 2.8]{GL22}, yields 
	  \[
	  u\geq \rho + \alpha -C_2 \varepsilon \geq v +\alpha - C_3\varepsilon. 
	  \] 
	  Let $(x_j)$ be a sequence such that $\lim_j (v-u)(x_j) = \sup_X (v-u)$. Then the above estimate gives 
	  \[
	  \sup_X (v-u) \leq  \frac{\sup_X (v-u) -\sup_X (u-v)}{2} + C_3 \varepsilon,
	  \]
	  which implies 
	  \[
	  \sup_X (v-u)  + \sup_X (u-v) \leq 2C_3 \varepsilon. 
	  \]
	  From the normalization $\sup_X u =\sup_X v=0$, we infer that $\sup_X (u-v) \geq 0$ and $\sup_X (v-u) \geq 0$. Thus the above inequality gives  $\sup_X(u-v) \leq 2C_3 \varepsilon$ and  $\sup_X(v-u) \leq 2C_3\varepsilon$, hence $\sup_X |u-v|\leq 2C_3\varepsilon$, finishing the proof of the first step. 
	  
	  In the second step, we solve 
	  \[
	  (\omega+dd^c \psi)^n = e^{\psi-u} g\omega^n, \; \psi\in \PSH(X,\omega) \cap L^{\infty}(X). 
	  \]
	 Since $(\omega+dd^c u)^n =e^{u -u} f\omega^n$, it follows from Theorem \ref{thm: stability exp} that 
	  \[
	  \sup_X |\psi-u| \leq C \|e^{-u/n}(f^{1/n} -g^{1/n})\|_{np},
	  \]
	  for some uniform constant $C>0$. Since $|u|\leq C_0$, we deduce from the above estimate that $\sup_X |\psi-u| \leq C_4 \|f^{1/n}-g^{1/n}\|_{np}$. Using this and  $\sup_X u=0$ we obtain $|\sup_X \psi | \leq C_5\|f^{1/n}-g^{1/n}\|_{np}$. Applying the first step for $\psi-\sup_X \psi$ and $v$ with $\varepsilon = C_4 \|f^{1/n}-g^{1/n}\|_{np}$, we obtain 
	  \[
	 \sup_X|\psi-v| \leq C_6 \|f^{1/n}-g^{1/n}\|_{np}.
	  \]
	  The triangle inequality then finishes the proof. 

\end{proof}

We now present a new proof of Yau's $L^{\infty}$
  estimate, utilizing the resolution of \eqref{eq: MA compact intro}. The approach closely mirrors the method outlined in Theorem \ref{thm: Linfty est local}.

\begin{theorem}\label{thm: Linfty global}
    Let $\mu$ be a probability measure on $X$ that vanishes on pluripolar sets. Assume that, for all $u\in \PSH(X,\omega)$ with $\sup_X u=0$, we have
    \begin{equation}\label{eq: Skoda-Zeriahi}
        \int_X e^{-2\gamma u} d\mu \leq A_{\mu},
    \end{equation}
    where $\gamma\leq n$ is a positive constant. Let $\varphi\in \PSH(X,\omega)$, $\sup_X \varphi=0$, solve the Monge-Amp\`ere equation $(\omega+dd^c \varphi)^n = \mu$. Then  
    \[
    \varphi  \geq -1- \frac{n \log n+\log A_{\mu} - n \log \gamma}{\gamma}.
    \]
\end{theorem}

Note that by the uniform version of Skoda's inequality by Zeriahi \cite{Zer01}, any positive measure with $L^p$ density ($p>1$) satisfies \eqref{eq: Skoda-Zeriahi}.
\begin{proof}
Consider $\nu = e^{-\gamma \varphi} \mu$. Then, by H\"older's inequality, we have, for all $u\in \PSH(X,\omega)$ with $\sup_X u=0$,
\[
\int_X e^{-\gamma u} d\nu  = \int_X e^{-\gamma u} e^{-\gamma \varphi} d\mu \leq A_{\mu}. 
\]
It follows from \cite[Theorem 11.13]{GZbook} that there exists $u \in \mathcal{E}^1(X,\omega)$ with $\sup_X u=0$ solving 
\[
(\omega+dd^c u)^n = \frac{e^{-\gamma u} \nu}{\int_X e^{-\gamma u} d\nu}. 
\]
The function $v= e^{\gamma u/n}-1$ is $\omega$-psh, $-1\leq v\leq 0$, and it satisfies 
\begin{flalign*}
(\omega+dd^c v)^n &\geq \gamma^n n^{-n} e^{\gamma u} (\omega+dd^c u)^n \\
& \geq \gamma^n n^{-n} A_{\mu}^{-1}  e^{\gamma (v-\varphi)}(\omega+dd^c \varphi)^n\\
& = e^{\gamma (v-C-\varphi)} (\omega+dd^c \varphi)^n,
\end{flalign*}
where $C$ is a constant verifying $e^{-\gamma C} =\gamma^n n^{-n} A_{\mu}^{-1}$.
By the domination principle, \cite[Corollary 2.9]{GL22},  we thus get  
\[
\varphi \geq v-C \geq -1- \frac{n \log n+\log A_{\mu} - n \log \gamma}{\gamma}. 
\]
\end{proof}

We are now ready to prove the main result of this section, addressing the uniqueness of the solution of \eqref{eq: MA compact intro} for sufficiently small $\gamma>0$. 
\begin{theorem}
	Assume $f$ is a probability density on $X$ which belongs to $L^p(X,\omega^n)$, for some $p>1$. Then there exists $\gamma_0>0$ such that for all $\gamma \in (0,\gamma_0)$, the equation 
	\[
	(\omega+dd^c u)^n =e^{-\gamma u} f\omega^n,
	\]
	has a unique solution $u\in \PSH(X,\omega) \cap L^{\infty}(X)$. 
\end{theorem}

\begin{proof}
We first discuss the boundedness of the solutions. 
By H\"older's inequality and the Skoda-Zeriahi estimate \cite{Zer01}, we can find $\beta>0$ and $A_{\beta}>0$ such that \eqref{eq: Skoda-Zeriahi} holds for $\mu =f\omega^n$, i.e. 
\[
\int_X e^{-2\beta u} f\omega^n \leq A_{\beta},\; \forall u \in \PSH(X,\omega), \; \sup_X u=0. 
\]
 Fixing $\alpha\in (0,\beta)$, as shown in  \cite[Theorem 11.13]{GZbook} using the variational approach from \cite{BBGZ13}, there exists $\varphi\in \mathcal{E}^1(X,\omega)$, solving
\[
(\omega+dd^c \varphi)^n = \frac{e^{-\alpha \varphi} f\omega}{\int_X e^{-\alpha \varphi} f\omega^n}.
\]
By H\"older's inequality, we have, for all $u\in \PSH(X,\omega)$ with $\sup_X u=0$,
\[
\int_X e^{-\alpha u} (\omega+dd^c \varphi)^n \leq \int_X e^{-\alpha u}e^{-\alpha \varphi} f\omega^n \leq A_{\beta}.
\]
It thus follows from Theorem \ref{thm: Linfty global} that $\varphi$ is bounded. 

To address uniqueness, we assume $u,v$ are two bounded solutions and we write $u_0= u-\sup_X u$, $v_0=v-\sup_X v$. We rewrite the Monge-Amp\`ere equations of $u$ and $v$ as 
	\[
	(\omega+dd^c u_0)^n = e^{-\gamma u_0 - b_1} f\omega^n, \; (\omega+dd^c v_0)^n = e^{-\gamma v_0 - b_2} f\omega^n. 
	\]
	Then the constants $b_1,b_2$ can be computed as 
	\[
	e^{b_1} = \int_X e^{-\gamma u_0} f\omega^n, \; e^{b_2} = \int_X e^{-\gamma v_0} f\omega^n. 
	\] 
    We have $|b_1|\leq C_1$ and $|b_2|\leq C_1$, hence 
	\[
	|b_1-b_2|\leq C_2 |e^{b_1}-e^{b_2}| \leq \gamma C_3 \sup_X |u_0-v_0|.
	\]
	From the stability result, Theorem \ref{thm:stability},  we obtain 
	\begin{flalign*}
		\sup_X |u_0-v_0| &\leq C \sup_X |e^{-(\gamma u_0+b_1)/n} - e^{-(\gamma v_0+b_2)/n} | \|f\|_p^{1/n}\\
	& \leq 	C_4 \sup_X |\gamma (u_0-v_0) + (b_1-b_2)| \\
	&\leq \gamma C_5 \sup_X |u_0-v_0|. 
	\end{flalign*}
	Choosing $\gamma$ so small that $\gamma C_5<1$, we obtain $u_0=v_0$, hence $u=v$. 
\end{proof}

The following example is well-known (see Exercise 11.13 in \cite{GZbook}). 
\begin{example}\label{ex: FS}
    Let $\omega= \omega_{FS}$ be the Fubini-Study metric on $X=\mathbb{P}^n$. In the local chart $\mathbb{C}^n\subset X$ its local expression is 
    \[
    \omega = dd^c \log (1+ \|z\|^2). 
    \]
    For $\varepsilon>0$ we consider the $\omega$-psh function on $X$ whose local expression in $\mathbb{C}^n$ is 
    \[
    u_{\varepsilon} = \log (\|z\|^2 +\varepsilon) - \log  (\|z\|^2+1). 
    \]
    A direct computation shows that 
    \[
    (\omega+dd^c u_{\varepsilon})^n = (dd^c (\log \|z\|^2+\varepsilon))^n = C e^{-(n+1) u_{\varepsilon}} \omega^n,
    \]
    where $C$ is a positive constant. Computing the total mass of both sides we obtain
\[
C= \frac{\int_X \omega^n }{\int_X e^{-(n+1) u_{\varepsilon}} \omega^n}.
\]
The equation 
    \[
    (\omega+dd^c u)^n= e^{-(n+1) u} \omega^n
    \]
    admits many solutions since $u_{\varepsilon}+C_{\varepsilon}$ solves it for a suitable constant $C_{\varepsilon}$. 
\end{example}

\section{Monge-Amp\`ere equations on compact Hermitian manifolds}\label{sect: hermitian}
Let $(X,\omega)$ be a compact Hermitian manifold of dimension $n$. We let $\PSH_0(X,\omega)$ denote the set of $\omega$-psh functions $u$ normalized by $\int_X u \omega^n=0$. 

Following Yau's resolution of the Calabi conjecture, the study of complex Monge-Amp\`ere equations on compact Hermitian manifolds has drawn considerable attention. % notably through the works \cite{Cher87}, \cite{GL10}, \cite{TW10}. 
The equation we focus on is
\begin{equation}\label{eq: MA Hermitian} (\omega+dd^c \varphi)^n = e^{-\gamma \varphi} f\omega^n, 
\end{equation}
where $\gamma$ is a constant and $f$ is a probability density function belonging to $L^p(X,\omega^n)$ for some $p>1$.
When $\gamma<0$ and $f>0$ is smooth, Cherrier \cite{Cher87} proved the existence of a unique smooth solution. The case $\gamma=0$ was addressed by Tosatti and Weinkove \cite{TW10}, following contributions from Guan and Li \cite{GL10}. In this case, the equation is solvable modulo a multiplicative positive constant. 

A pluripotential theory tailored to this equation has since been developed by S.Dinew, Ko{\l}odziej, Nguyen, Guedj, T\^o, the present authors, and many others. A notable result is the resolution of \eqref{eq: MA Hermitian} for $\gamma\leq 0$ and $f\in L^p$, $p>1$. Existence of solutions in this case is well established but uniqueness (for $\gamma=0$) is only known when $f$ is strictly positive, as shown in \cite{KN19} with simplification by \cite{LPT21}.  We direct the reader to \cite{Diw16AFST}, \cite{KN19}, \cite{GL21c}, \cite{GL22}, and the references therein for further details. 

In this section, we study solutions of \eqref{eq: MA Hermitian} with $\gamma>0$. Fixing $1<q<p$, it follows from H\"older's inequality that 
\begin{equation}
    \label{eq: Skoda-Zeriahi Hermitian}
    \int_X e^{-\gamma qu} f^qdV <+\infty, \, \forall u \in \PSH_0(X,\omega),
\end{equation}
when $\gamma$ is sufficiently small. Let $\gamma_0$ be the supremum of all $\gamma>0$ such that \eqref{eq: Skoda-Zeriahi Hermitian} holds. 

\begin{theorem}\label{thm: MA Hermitian}
   For all $\gamma<\gamma_0$, \eqref{eq: MA Hermitian} admits a continuous solution. 
\end{theorem}
\begin{proof}
   The continuity of solutions follows from \cite{KN19}. We will use the fixed point theorem to solve the equation. The starting point of this method is that the set $\PSH_0(X,\omega)\subset L^1(X,\omega^n)$, of all $\omega$-psh functions $u$ normalized by $\int_X u \omega^n=0$, is convex and compact with respect to the $L^1$-topology. Fix $\gamma\in (0,\gamma_0)$,  $\alpha>0$ so small that $\alpha +\gamma<\gamma_0$. 
For each $u\in \PSH_0(X,\omega)$, we solve 
\begin{equation}\label{eq: MA hermitian proof}
    (\omega+dd^c v)^n = e^{\alpha v- (\alpha+\gamma) u} fdV, \; v\in \PSH(X,\omega)\cap L^{\infty}(X).
\end{equation}
The existence and uniqueness of $v$ follow from \cite{Ng16AIM}. Moreover, $v$ is bounded because $e^{-(\alpha+\gamma) u} f$ belongs to $L^r$ for some $r>1$ (again by H\"older's inequality). 
We define the map 
\begin{eqnarray*}
F :  \PSH_0(X,\omega) & \longrightarrow & \PSH_0(X,\omega)\\
    u &\mapsto & v-\int_X v \omega^n.
\end{eqnarray*}
We prove that $F$ is continuous with respect to the $L^1$-topology. For this, let $(u_j)$ be a sequence in $\PSH_0(X,\omega)$ converging in $L^1$ to $u\in \PSH_0(X,\omega)$, and we prove that $F(u_j)$ converges in $L^1$ to $F(u)$. Observe that, by H\"older's inequality, $e^{-(\alpha+\gamma) u_j}f$ converges in $L^r$ to $e^{-(\alpha+\gamma) u}f$ for some $r>1$. It thus follows from \cite{Ng16AIM} that $v_j$ uniformly converges to $v\in \PSH_0(X,\omega)\cap L^{\infty}$. By the continuity of the Monge-Amp\`ere operator, \cite[Theorem 3.18]{GZbook}), it follows that $v$ solves \eqref{eq: MA hermitian proof}. Thus $F(u_j) \to F(u)$ as $j\to +\infty$, proving the continuity of $F$. 

By Schauder's fixed point theorem, see \cite[Theorem B.2, page 302]{Tay11}, there exists a fixed point of $F$, say $u\in \PSH_0(X,\omega)$. Then $u$ is bounded and $u+C$ solves \eqref{eq: MA Hermitian} for some constant $C$. 
\end{proof}

We next state the following stability result, refining \cite{KN19} and \cite{LPT21}. 

\begin{theorem}\label{thm: stability Hermitian}
    Assume $f,g$ are probability measures and $\min(f,g)\geq c_0>0$, where $c_0$ is a constant. Let $u,v \in \PSH(X,\omega)$, $\sup_X u =\sup_X v=0$, solve 
    \[
    (\omega+dd^c u)^n =  e^{c_f} f\omega^n, \; (\omega+dd^c v)^n =e^{c_g} g\omega^n,
    \]
    where $c_f$ and $c_g$ are constants. 
    Then
    \[
    \sup_X |u-v| + |c_f-c_g| \leq C \|f^{1/n}-g^{1/n}\|_{np}, 
    \]
    where $C$ is a positive constant depending on $c_0$, $p$, $X$, $\omega$, and an upper bound for $\|f\|_p$, $\|g\|_p$.  
\end{theorem}
\begin{proof}
Since the proof follows the same arguments as in \cite{LPT21}, with only minor adjustments from Section \ref{sect: Kahler}, we leave the details to the interested reader.
\end{proof}

\begin{remark}
    When $f>0$ is smooth and $\omega$ is K\"ahler, the solutions to \eqref{eq: MA Hermitian} obtained via Theorem \ref{thm: MA Hermitian} are smooth, as demonstrated in \cite{Szekelyhidi_Tosatti_2011_APDE} using the Monge-Amp\`ere flow. 
    In the Hermitian context, when $f>0$ is smooth and $\omega$ is positive but not necessarily K\"ahler, the approach of \cite{Szekelyhidi_Tosatti_2011_APDE} was adapted to the Hermitian context in \cite{Nie14}, where the same result was achieved under a curvature condition on $\omega$. This condition was later removed in \cite[Theorem C]{KN19} by applying the stability result obtained there. 
\end{remark}

\begin{theorem}\label{thm: uniqueness Hermitian}
Let $f\geq c_0>0$ be a probability density function belonging to $L^p$ for some $p>1$.
    There exists a positive constant $\gamma_1>0$ such that for all $\gamma <\gamma_1$, the equation \eqref{eq: MA Hermitian} has a unique bounded solution.
\end{theorem}
\begin{proof}
    Assume $u,v \in \PSH(X,\omega)\cap L^{\infty}$ are solutions of 
    \[
    (\omega+dd^c u)^n = e^{-\gamma u +a_1} fdV,\; (\omega+dd^c v)^n = e^{-\gamma v +a_2} fdV, \; \sup_X u=\sup_X v=0. 
    \]
    We rewrite the equations as follows
    \[
    (\omega+dd^c u)^n = e^{c_g} gdV,\; (\omega+dd^c v)^n = e^{c_h} hdV, \; \sup_X u=\sup_X v=0,
    \]
    where $g= e^{-\gamma u} f$ and $h=e^{-\gamma v}f$. For $\gamma\leq \gamma_0$, the densities $g,h$ are uniformly bounded in $L^q$, for some $q>1$, and moreover, $\min(g,h)\geq c_0$. The uniform estimate (see  \cite[Theorem 2.2]{GL21c}) gives $|u|+|v|\leq C_1$. Therefore, applying Theorem \ref{thm: stability Hermitian}, we obtain 
    \[
    \sup_X |u-v|\leq C_2 \|g^{1/n}-h^{1/n}\|_{np} \leq C_3 \gamma \sup_X |u-v|.
    \]
   Thus for $\gamma<C_3^{-1}$, the above inequality yields $u=v$, finishing the proof. 
\end{proof}

\bibliographystyle{alpha}
\bibliography{Biblio}

\end{document}